\documentclass[11pt]{amsart}
\usepackage[margin=1in]{geometry}
\usepackage{amssymb,amsmath,mathrsfs}
\usepackage{enumerate,tikz}
\usepackage{hyperref}
\usetikzlibrary{patterns}
\usepackage{xstring}

\newtheorem{theorem}{Theorem}[section]
\newtheorem{corollary}[theorem]{Corollary}
\newtheorem{proposition}[theorem]{Proposition}
\newtheorem{lemma}[theorem]{Lemma}

\theoremstyle{definition}
\newtheorem*{rem}{Remark}

\numberwithin{equation}{section}

\tikzstyle{mesh}=[pattern=north east lines, pattern color=gray!50, draw=gray!60]

\def\oeis#1{\cite[#1]{Sloane}}
\def\abs#1{\big\lvert#1\big\rvert}
\def\Av{\mathcal{S}}
\def\Patt{\mathcal{P}}
\def\A{\mathcal{A}}
\def\T{\mathcal{T}}
\def\U{\mathcal{U}}
\def\V{\mathcal{V}}

\DeclareMathOperator\red{red}

\DeclareMathOperator\ins{ins}

\newcommand{\perm}[1]{%
	\def\Perm{#1}
	\StrSubstitute{\Perm}{,}{\,}
}

\begin{document}
\title[On separable permutations and other pairs in the Schr\"oder class]{On separable permutations and three other pairs in the Schr\"oder class}
\author[J.~Gil]{Juan B. Gil}
\address{Penn State Altoona\\ 3000 Ivyside Park\\ Altoona, PA 16601}
\email{jgil@psu.edu}

\author[O.~Lopez]{Oscar A. Lopez}
\address{Penn State Harrisburg\\ 777 West Harrisburg Pike\\ Middletown, PA 17057}
\email{oal5053@psu.edu}

\author[M.~Weiner]{Michael D. Weiner}
\address{Penn State Altoona\\ 3000 Ivyside Park\\ Altoona, PA 16601}
\email{mdw8@psu.edu}


\begin{abstract}
We study positional statistics for four families of pattern-avoiding permutations counted by the large Schr\"oder numbers. Specifically, we focus on the pairs of patterns $\{2413,3142\}$ (separable permutations), $\{1324,1423\}$, $\{1423,2413\}$, and $\{1324,2134\}$. For each class, we derive multivariate generating functions that track the relative positions of specific entries. Our approach combines structural decompositions with the kernel method to obtain explicit formulas involving the generating function for the Schr\"oder numbers. As a byproduct, we obtain alternative proofs that each of these classes is enumerated by the Schr\"oder numbers. We also identify several known triangular arrays arising from our positional refinements, including connections to the central binomial coefficients and sequences appearing in the work of Kreweras on covering hierarchies.
\end{abstract}

\maketitle

\section{Introduction}
\label{sec:intro}

Before stating our results, let us recall some standard terminology. A \emph{pattern} of \emph{size} $k$ is a permutation on $\{1,\dots,k\}$, written in one-line notation. A permutation $\sigma$ on $\{1,\dots,n\}$ \emph{contains} the pattern $\pi$ if some subsequence of $\sigma$ has the same relative order as $\pi$; if no such subsequence exists, we say that $\sigma$ \emph{avoids} $\pi$. For example, the permutation $\sigma=\perm{4,6,2,1,5,3}$ contains three instances of $321$ (in the subsequences $(4,2,1)$, $(6,2,1)$, and $(6,5,3)$), but avoids the pattern $123$. For a set $\Patt$ of patterns, we let $\Av_n(\Patt)$ denote the set of permutations of size $n$ that avoid every pattern in $\Patt$, and write $\Av(\Patt)$ for the union of these sets over all $n$. The symmetries of the square, generated by the reverse, complement, and inverse operations, act on permutations and hence on sets of patterns; the orbit of a pattern set under this action is its \emph{symmetry class}. Since each of these symmetries is a bijection that preserves pattern containment, all pattern sets in the same symmetry class have equinumerous avoidance classes.

A permutation is called \emph{separable} if it avoids both $2413$ and $3142$. West~\cite{West95} showed that separable permutations are counted by the \emph{large Schr\"oder numbers}, 
\[ 1, 2, 6, 22, 90, 394, 1806, 8558, 41586, 206098,\dots \text{ (\oeis{A006318})}, \] 
which are known to enumerate, for instance, the lattice paths from $(0,0)$ to $(2n,0)$ that use the steps $(1,1)$, $(1,-1)$, and $(2,0)$ and never go below the horizontal axis. The pair $\{2413,3142\}$ is not alone in this respect: there are exactly ten symmetry classes of pattern pairs of size 4 whose avoidance classes are counted by these numbers (see Kremer~\cite{Kremer2000} and Kremer--Shiu~\cite{KremerShiu03}).

In previous work~\cite{GLW24}, we introduced a type of statistics, which we called \emph{positional statistics}, to study the class of $1324$-avoiding permutations. The main idea is to track the relative positions of specific entries in a permutation, such as the distance between the smallest and largest elements, or the position of the minimum. As a proof of concept, in this paper, we apply various types of positional statistics to four families of pattern-avoiding permutations counted by the large Schr\"oder numbers. We focus on $\Av(2413,3142)$ (separable permutations), and the classes $\Av(1324,1423)$, $\Av(1423,2413)$, and $\Av(1324,2134)$.

For each class, we derive multivariate generating functions that refine the enumeration according to positional constraints. As a byproduct, we obtain alternative proofs for the total enumeration of these  classes. Our approach combines structural decompositions of the permutations with the kernel method to solve the resulting functional equations.

Some of the structural decompositions used in this paper share the general spirit of the insertion encoding of Albert, Linton, and Ru\v{s}kuc~\cite{ALR05}. In that framework, permutations are built by successively inserting new maximum elements into slots, and pattern classes whose encodings form regular or context-free languages admit systematic enumeration. While the total enumeration of the four classes considered here may fall within the scope of that framework, the insertion encoding tracks slot configurations rather than positional relationships between specific entries such as the distance between the minimum and maximum entries, the position of the minimum within skew-indecomposable elements, or the value of the last entry. Such refinements are not explicitly part of that framework, and they constitute the primary contribution of this paper.

The paper is organized as follows. In Section~\ref{sec:separable}, we study the class $\Av(2413,3142)$ of separable permutations. We derive a generating function that tracks the position of the minimum and use it to give an alternative proof that separable permutations are counted by the large Schr\"oder numbers. We also obtain a multivariate generating function that tracks the distance between the maximum entry of a permutation and the smallest entry to its left (in one-line notation). When tracking the position of the $1$ or the positive distance between $1$ and the maximum, our results lead to triangular arrays not currently listed in the OEIS. We leave it to the interested reader to explore possible connections to Schr\"oder paths or other combinatorial families counted by the little or large Schr\"oder numbers. 

In Section~\ref{sec:1324_1423}, we turn to $\Av(1324,1423)$ and enumerate its skew-indecomposable elements by the position of the minimum. The resulting triangular array connects to work by Kreweras~\cite{Kreweras73} on covering hierarchies, and the generating function yields the little Schr\"oder numbers, from which the large Schr\"oder enumeration follows.

Section~\ref{sec:1423_2413} addresses the class $\Av(1423,2413)$, where we enumerate permutations with $1$ to the left of $n$ according to the position of $n$. The resulting triangle relates to another array studied by Kreweras, and we recover again the Schr\"oder enumeration as a consequence.

Finally, in Section~\ref{sec:1324_2134}, we study $\Av(1324,2134)$ using a different positional statistic: the value of the last entry. We derive the corresponding generating function and provide yet another alternative proof of the Schr\"oder enumeration. We also use the reverse-complement symmetry to connect with $\Av(1243,1324)$ and obtain a refinement by the distance between entries $1$ and $n$, which leads to central binomial coefficients and the triangle \oeis{A092392}.

Beyond the individual results, the broader point of this paper is that positional statistics provide a systematic framework for extracting finer enumerative information from pattern-avoiding permutation classes. The alternative proofs of the Schr\"oder enumeration were not intended as a primary goal but rather a consequence of the refined recurrence relations and generating functions. The specific classes we chose were motivated by our previous work~\cite{GLW24} on $\Av(1324)$, and by our independent interest in separable permutations. We believe this approach has potential beyond the classes studied here, particularly for avoidance classes where existing enumeration techniques do not yield combinatorial decompositions.

For more on pattern-avoiding permutations, we refer to the book by Kitaev~\cite{Kitaev11}.

\section{Separable permutations} 
\label{sec:separable}

Let us start by setting up some notation. For $a,k\ge 1$ and a set of patterns $\Patt$, we let $\Av_{n,k}^{a\prec n}(\Patt)$ be the set of $\Patt$-avoiding permutations $\sigma\in \Av_n(\Patt)$ such that:
\smallskip
\begin{itemize}
\itemsep5pt
\item $\sigma^{-1}(n)-\sigma^{-1}(a)=k$,
\item $\sigma^{-1}(b)-\sigma^{-1}(n)>0$ for every $b\in\{1,\dots,a-1\}$.
\end{itemize}
In other words, in one-line notation, every $\sigma\in \Av_{n,k}^{a\prec n}(\Patt)$ has entry $a$ to the left of $n$ at distance $k$, and all the entries less than $a$ are to the right of $n$. 

Using generating trees, West \cite{West95} showed that the class $\Av(2413,3142)$ of separable permutations is counted by the large Schr\"oder numbers (\oeis{A006318}). In this section, we will provide an alternative proof of this result using positional statistics.

Let
\[ S(x) = 1 + x + 2x^2 + 6x^3 + 22x^4 + 90x^5 + 394x^6 +\cdots \] 
be the generating function for the sequence $a_0=1$, $a_n = \abs{\Av_{n}(2413,3142)}$.

Let $\Av_n^{\ell\mapsto 1}(2413,3142)$ be the subset of permutations having entry $1$ at position $\ell$:
\[ \Av_n^{\ell\mapsto 1}(2413,3142) = \{\sigma\in\Av_n(2413,3142): \sigma(\ell) = 1\}. \]

\begin{proposition} \label{prop:byPos1}
The function $g(x,u)=\sum\limits_{n=1}^\infty\sum\limits_{\ell=1}^n\, \abs{\Av_n^{\ell\mapsto 1}(2413,3142)}\, u^{\ell} x^n$ satisfies 
\[ g(x,u) =\frac{xuS(x)S(xu)}{S(x)+S(xu)-S(x)S(xu)}. \]
A few terms of $\abs{\Av_n^{\ell\mapsto 1}(2413,3142)}$ are listed in Table~\ref{tab:byPos1}.
\end{proposition}

\begin{proof}
Let $g_i(x,u)$ and $g_d(x,u)$ be the components of $g(x,u)$ counting, respectively, the corresponding indecomposable and decomposable permutations of size greater than 1. Thus,
\begin{align*}
 g_i(x,u) &= u^2x^2 + (u^2+2u^3)x^3 +\cdots, \\
 g_d(x,u) &= ux^2 + (2u+u^2)x^3 +\cdots, 
\end{align*}
and $g(x,u) = xu + g_i(x,u) + g_d(x,u)$. Note that, since every permutation is either indecomposable or has an indecomposable factor, we have
\begin{equation} \label{eq:gxu}
  g(x,u) = (xu + g_i(x,u))S(x) \;\text{ and }\; g_d(x,u) = (xu + g_i(x,u))(S(x)-1).
\end{equation}
Since the reverse map is an involution on $\Av_{n}(2413,3142)$, we have 
\[ g(x,u) = u g(xu,\tfrac1u) = (xu + ug_i(xu,\tfrac1u))S(xu). \]
Moreover, since the reverse of an indecomposable separable permutation is decomposable (this follows from the fact that every separable permutation of size at least 2 is either a direct sum or a skew sum of smaller separable permutations), we also have $g_d(x,u) = ug_i(xu,\frac1u)$, and therefore
\begin{align*}
  g(x,u) &= (xu + g_d(x,u))S(xu) \\
  &= \big(xu + (xu + g_i(x,u))(S(x)-1)\big)S(xu) \\
  & = xu S(x)S(xu) + g_i(x,u)(S(x)-1)S(xu).
\end{align*}
Combining this with \eqref{eq:gxu}, we arrive at the equation
\[ (xu + g_i(x,u))S(x) = xu S(x)S(xu) + g_i(x,u)(S(x)-1)S(xu), \]
which gives
\[ g_i(x,u) = \frac{xu S(x)(S(xu) - 1)}{S(x) + S(xu) - S(x)S(xu)}. \]
Using again \eqref{eq:gxu}, we then get
\begin{align*}
 g(x,u) &=  (xu + g_i(x,u))S(x) \\
 &= xuS(x) +  \frac{xu S(x)^2(S(xu) - 1)}{S(x) + S(xu) - S(x)S(xu)} \\
 &= \frac{xu S(x)S(xu)}{S(x) + S(xu) - S(x)S(xu)},
\end{align*}
as claimed.
\end{proof}

\begin{table}[ht]
\begin{tabular}{c|rrrrrrrr||r}
\rule[-5pt]{0pt}{0pt} $n \backslash \ell$ & 1 & 2 & 3 & 4 & 5 & 6 & 7 & 8 & $\Sigma$ \\ \hline
1 & 1 & & & & & & & & \rule{0pt}{12pt} 1 \\
2 & 1 & 1 & & & & & & & 2 \\
3 & 2 & 2 & 2 & & & & & & 6 \\
4 & 6 & 5 & 5 & 6 & & & & & 22 \\
5 & 22 & 16 & 14 & 16 & 22 & & & & 90 \\
6 & 90 & 60 & 47 & 47 & 60 & 90 & & & 394 \\
7 & 394 & 248 & 180 & 162 & 180 & 248 & 394 & & 1806 \\
8 & 1806 & 1092 & 752 & 629 & 629 & 752 & 1092 & 1806 & 8558
\end{tabular}
\bigskip
\caption{Triangle for $\abs{\Av_n^{\ell\mapsto 1}(2413,3142)}$ from Proposition~\ref{prop:byPos1}.}
\label{tab:byPos1}
\end{table}

\begin{corollary}
We have $S(x) = \tfrac12 \big(3 - x - \sqrt{x^2 - 6x + 1}\,\big)$, which is the generating function for the sequence of large Schr\"oder numbers.
\end{corollary}
\begin{proof}
Note that $g(x,1) = S(x) - 1$. Therefore, 
\[ S(x) - 1 = \frac{xS(x)S(x)}{S(x)+S(x)-S(x)S(x)} = \frac{xS(x)}{2-S(x)}, \]
and so $S(x)^2 - (3-x)S(x) + 2 = 0$. Hence $S(x) = \tfrac12 \big(3 - x - \sqrt{x^2 - 6x + 1}\,\big)$.
\end{proof}

\bigskip
We proceed to enumerate the elements of $\Av_{n,k}^{a\prec n}(2413,3142)$, starting with the case $a=1$.

\begin{table}[ht]
\begin{tabular}{c|rrrrrrr||r}
\rule[-5pt]{0pt}{0pt} $n \backslash k$ & 2 & 3 & 4 & 5 & 6 & 7 & 8& $\Sigma$ \\ \hline
2 & 1 & & & & & & & \rule{0pt}{12pt} 1 \\
3 & 2 & 1 & & & & & & 3 \\
4 & 5 & 4 & 2 & & & & & 11 \\
5 & 16 & 13 & 10 & 6 & & & & 45 \\
6 & 60 & 46 & 37 & 32 & 22 & & & 197 \\
7 & 248 & 180 & 140 & 125 & 120 & 90 & & 903 \\
8 & 1092 & 760 & 567 & 490 & 480 & 496 & 394 & 4279 
\end{tabular}
\bigskip
\caption{Triangle for $\abs{\Av_{n,k}^{1\prec n}(2413,3142)}$ from Proposition~\ref{prop:sep_1<n}.}
\label{tab:sep_1<n}
\end{table}

\begin{proposition}
\label{prop:sep_1<n}
The function $f(x,t) = \!\sum\limits_{n,k\ge 1}\, \abs{\Av_{n,k}^{1\prec n}(2413,3142)} t^k x^n$ satisfies
\[ f(x,t) = \frac{x^2t S(x)S(xt)^2}{\big(S(xt) + S(x) - S(xt)S(x)\big)^2}. \]
\end{proposition}
\begin{proof}
Every $\sigma\in \Av_{n,k}^{1\prec n}(2413,3142)$ is decomposable (since $1$ is to the left of $n$) and must be of the form $\sigma=\pi_1\oplus \pi_2\oplus \pi_3$, where $\pi_1$ is the indecomposable component containing $1$, $\pi_3$ is the indecomposable component containing $n$, and $\pi_2$ is a separable permutation (possibly empty). The distance between the entries 1 and $n$ is given by 
\[ k = (|\pi_1|-\ell_1) + |\pi_2| + \ell_n, \] 
where $\ell_1$ is the position of the 1 in $\pi_1$ and $\ell_n$ is the position of the largest element of $\pi_3$. Note that $\ell_n$ is the position of the 1 in the complement of $\pi_3$.

Since $t$ tracks the distance $k$, the contribution of $\pi_2$ to the generating function is $S(xt)$ (each entry contributes one unit of distance), the contribution of $\pi_3$ is $xt + g_i(x,t)$ (where $t$ marks the position $\ell_n$ of the largest element), and the contribution of $\pi_1$ is $x + g_i(xt, \tfrac{1}{t})$ (where $t^{|\pi_1|}$ accounts for the size and $t^{-\ell_1}$ adjusts for the position of the $1$). Therefore,
\[ f(x,t) = \big(x + g_i(xt,\tfrac1t)\big)S(xt)\big(xt + g_i(x,t)\big). \]
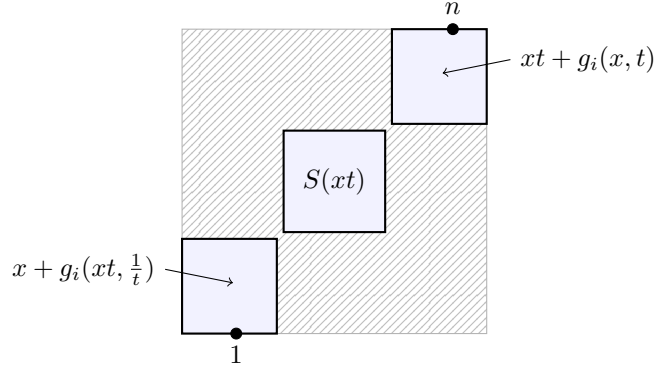
\begin{center}
\begin{tikzpicture}[scale=0.9]
\draw[mesh] (0,0) rectangle (4.5,4.5);
\draw[thick,fill=blue!5] (0,0) rectangle (1.4,1.4);
\draw[thick,fill=blue!5] (1.5,1.5) rectangle (3,3);
\draw[thick,fill=blue!5] (3.1,3.1) rectangle (4.5,4.5);
\draw[fill] (0.8,0) circle(0.08);
\draw[fill] (4,4.5) circle(0.08);
\node[below=1pt] at (0.8,0) {\small $1$};
\node[above=2pt] at (4,4.5) {\small $n$};
\draw[<-] (0.75,0.75) to +(-1,0.2) node[left]{\small $x+g_i(xt,\tfrac1t)$}; 
\node at (2.25,2.25) {\small $S(xt)$}; 
\draw[<-] (3.85,3.85) to +(1,0.2) node[right]{\small $xt+g_i(x,t)$};
\end{tikzpicture}
\end{center}

Using \eqref{eq:gxu}, this implies 
\[ f(x,t) = g(xt,\tfrac1t)\cdot \frac{g(x,t)}{S(x)}, \]
which by Proposition~\ref{prop:byPos1} becomes
\[ f(x,t) = \frac{x S(xt)S(x)}{S(xt) + S(x) - S(xt)S(x)} \cdot  \frac{xt S(xt)}{S(x) + S(xt) - S(x)S(xt)}. \]
This product simplifies to the claimed formula.
\end{proof}

Finally, observe that for $a>1$, any permutation $\sigma\in \Av_{n,k}^{a\prec n}(2413,3142)$ must be of the form $\sigma = \pi\ominus\tau$, where $\pi\in \Av_{m,k}^{1\prec m}(2413,3142)$ with $m=n-a+1$, and $\tau\in \Av_{a-1}(2413,3142)$. As a direct consequence, we obtain the following proposition.

\begin{proposition} 
If $F(x,t,s) = \sum\limits_{n,k,a\ge 1}\, \abs{\Av_{n,k}^{a\prec n}(2413,3142)} s^{a} t^k x^n$, then
\[ F(x,t,s) = f(x,t)\cdot sS(xs) =  \frac{x^2ts S(x) S(xt)^2 S(xs)}{\big(S(xt) + S(x) - S(xt)S(x)\big)^2}. \]
\end{proposition}

\section{(1324,1423)-avoiding permutations}
\label{sec:1324_1423}

A permutation $\sigma$ is called {\em skew-decomposable} if there are nonempty permutations $\pi$ and $\tau$ such that $\sigma = \pi\ominus\tau$. Otherwise, we say that $\sigma$ is {\em skew-indecomposable}. Let
\[ \A^{\textup{s-ind}}_{n,\ell} = \{\sigma\in\Av_n(1324,1423): \sigma \text{ is skew-indecomposable and } \sigma(\ell)=1\}. \]
Clearly, $\A^{\textup{s-ind}}_{1,1} = \{1\}$, $\A^{\textup{s-ind}}_{2,1} = \{12\}$, and $\A^{\textup{s-ind}}_{n,n} = \varnothing$ for every $n\ge 2$.

\smallskip
In this section, we will make use of the following notation. Given a permutation $\pi$ of size $n-1$, we let $\ins^k_j(\pi)$ be the permutation of size $n$ obtained by inserting $k$ into $\pi$ at position $j$. More precisely, $\ins^k_j(\pi)$ is constructed by
\begin{itemize}
\item increasing every entry $\ge k$ in $\pi$ by one, 
\item moving every entry at position $\ge j$ one unit to the right, and 
\item placing $k$ at position $j$. 
\end{itemize}

The recursive construction used below, based on inserting entries 1 and 2, shares the spirit of the insertion encoding framework studied in \cite{ALR05}; however, our focus is on tracking the position of the minimum, a refinement that goes beyond total enumeration.

\begin{proposition}\label{prop:1324_1423byPos1}
If $a_{n,\ell} = \abs{\A^{\textup{s-ind}}_{n,\ell}}$, then $a_{1,1} = 1$, $a_{2,1} = 1$, and for $n\ge 3$,
\begin{equation*}
a_{n,\ell} = 2a_{n-1,\ell} + \sum_{j=1}^{\ell-1} a_{n-1,j} \;\text{ for } 1\le \ell \le n-1.
\end{equation*}
Moreover, its generating function $g(x,u)=\sum\limits_{n=1}^\infty\sum\limits_{\ell=1}^n\, a_{n,\ell}\, u^{\ell} x^n$ satisfies
\[ g(x,u) = \frac{ux\Big[4u-3 + 4x - 3ux - \sqrt{1 - 6ux + u^2x^2}\Big]}{4\big(u-1-ux+2x\big)}. \]
\end{proposition}

\begin{proof}
Let $n\ge 3$ and suppose $\ell$ is such that $1\le \ell \le n-1$. 

For $1\le j<\ell$ every permutation $\tau_j$ in $\A^{\textup{s-ind}}_{n-1,j}$ gives rise to a unique permutation $\sigma$ in $\A^{\textup{s-ind}}_{n,\ell}$ obtained by inserting $1$ at position $\ell$. That is, $\sigma=\ins^1_\ell(\tau_j)$. Note that since $\tau_j(j)=1$, we have $\sigma(j)=2$, so entry $2$ is to the left of $1$ in $\sigma$. 

On the other hand, for $\ell < n-1$, every $\tau_\ell\in \A^{\textup{s-ind}}_{n-1,\ell}$ gives rise to two unique permutations $\sigma'$ and $\sigma''$ in $\A^{\textup{s-ind}}_{n,\ell}$ obtained by inserting $2$ at positions $\ell+1$ and $n$, respectively. In other words, $\sigma'=\ins^2_{\ell+1}(\tau_\ell)$ and $\sigma''=\ins^2_{n}(\tau_\ell)$. For example, if $\tau_\ell = \perm{2,5,1,4,3}$, then $\sigma'=\perm{3,6,1,2,5,4}$ and $\sigma''=\perm{3,6,1,5,4,2}$. Note that in these cases, entry $2$ appears to the right of $1$. Moreover, the ascent $12$ never happens at the end of $\sigma'$ or $\sigma''$ because $\tau_\ell$ is skew-indecomposable and cannot be decomposed as $\pi\ominus 1$. Also note that $a_{n-1,n-1}=0$ for $n\ge 3$.

All of the above insertions preserve the property of being skew-indecomposable and cannot create a $1324$ or $1423$ pattern unless the starting permutation of size $n-1$ already had any of these patterns. Finally, if entry $2$ is to the right of $1$ in $\sigma$, then it must be either adjacent to the $1$ or at the end of $\sigma$. Otherwise, a subsequence of the form $\perm{1,a,2,b}$ would create either a $1324$ (if $a<b$) or a $1423$ (if $a>b$) pattern. 

In conclusion, the set $\A^{\textup{s-ind}}_{n,\ell}$ is the disjoint union of the three subsets defined by the position of the $2$ relative to the position of the $1$, as described above. The subset of permutations where $2$ is at position $j<\ell$ is in bijection to $\A^{\textup{s-ind}}_{n-1,j}$, and the other two subsets (with $2$ to the right of $1$) are both in bijection to $\A^{\textup{s-ind}}_{n-1,\ell}$. Therefore,
\begin{equation*}
a_{n,\ell} = 2a_{n-1,\ell} + \sum_{j=1}^{\ell-1} a_{n-1,j} \;\text{ for } 1\le \ell \le n-1.
\end{equation*}

Using this recurrence relation and routine algebraic manipulations, we arrive at the functional equation
\[ (u - 1 - ux + 2x)g(x,u) = ux(1-x)(u-1) + ux g(ux,1). \]
Letting $u=\frac{2x-1}{x-1}$ (kernel method), we get $g\big(\tfrac{(2x-1)x}{x-1},1\big) = x$, and therefore
\[ g(z,1) = \frac{1+z-\sqrt{1-6z+z^2}}{4}. \]
As a consequence, we get
\begin{align*}
 g(x,u) &=  \frac{ux(1-x)(u-1) + ux g(ux,1)}{u - 1 - ux + 2x} \\
 &= \frac{ux\Big[4u-3 + 4x - 3ux -\sqrt{1-6ux+u^2x^2}\Big]}{4(u - 1 - ux + 2x)}. \qedhere
\end{align*}
\end{proof}

\begin{table}[t]
\begin{tabular}{c|rrrrrrr||r}
\rule[-5pt]{0pt}{0pt} $n \backslash \ell$ & 1 & 2 & 3 & 4 & 5 & 6 & 7 & $\Sigma$ \\ \hline
1 & 1 & & & & & & & \rule{0pt}{12pt} 1 \\
2 & 1 & & & & & & & 1 \\
3 & 2 & 1 & & & & & & 3 \\
4 & 4 & 4 & 3 & & & & & 11 \\
5 & 8 & 12 & 14 & 11 & & & & 45 \\
6 & 16 & 32 & 48 & 56 & 45 & & & 197 \\
7 & 32 & 80 & 144 & 208 & 242 & 197 & & 903
\end{tabular}
\bigskip
\caption{Triangle for $\{a_{n,\ell}\}$ from Proposition~\ref{prop:1324_1423byPos1}.}
\label{tab:1324_1423byPos1}
\end{table}

\begin{rem}
The triangular array in Table~\ref{tab:1324_1423byPos1} appears in work by Kreweras~\cite[p.~54]{Kreweras73} in the context of covering hierarchies of integer segments.
\end{rem}

\begin{corollary} \label{cor:Schroeder1324_1423}
If $G(x)$ is the generating function that counts the skew-indecomposable elements of $\Av_n(1324,1423)$, then
\[ G(x) = g(x,1) = \frac{1 + x - \sqrt{1 - 6x + x^2}}{4}. \]
This is the generating function for the little Schr\"oder numbers. Since the patterns $1324$ and $1423$ are both skew-indecomposable, we recover the known fact that the class $\Av_n(1324,1423)$ is enumerated by the large Schr\"oder numbers.
\end{corollary}

\section{(1423,2413)-avoiding permutations}
\label{sec:1423_2413}

The class $\Av(1423,2413)$ is known to be counted by the large Schr\"oder numbers. This was shown by Kremer~\cite{Kremer2000} using generating trees (see also Stankova~\cite{Stank94} and Kremer--Shiu~\cite{KremerShiu03}). In this section, we focus on the enumeration of $\Av_n^{1\prec n}(1423,2413)$ and use our results to provide an alternative proof of this fact.

The bijective maps $\phi_1$, $\phi_2$, $\phi_3$ constructed below share the spirit of insertion encoding~\cite{ALR05}, but are designed to track the position of $n$ relative to $1$, a positional refinement that goes beyond total enumeration.

Let $\A^{1\prec n}_{k\mapsto n}$ denote the set of permutations in $\Av_n(1423,2413)$ having entry $n$ at position $k$, and to the right of 1. That is,
\[ \A^{1\prec n}_{k\mapsto n} = \{\sigma\in\Av_n^{1\prec n}(1423,2413): \sigma(k)=n\}. \]

\begin{proposition} \label{prop:1423_2413by1nDist}
Let $2\le k\le n$. If $a_{n,k} = \abs{\A^{1\prec n}_{k\mapsto n}}$, then
\begin{align*}
 a_{n,2} &=1, \\
 a_{n,n} &= a_{n,n-1} + a_{n-1,n-1} \;\text{ for }\; n\ge 3, \\ 
 a_{n,k} &= a_{n,k-1} + a_{n-1,k} + a_{n-1,k-1} \;\text{ for }\; 3\le k < n.
\end{align*}
\end{proposition}
\begin{proof}
The only element of $\A^{1\prec n}_{2\mapsto n}$ is the permutation $1\,n\,(n-1)\cdots 2$, so $a_{n,2}=1$. 

For $k\ge 3$, we will give bijective maps to uniquely construct all the elements of $\A^{1\prec n}_{k\mapsto n}$ from the elements of $\A^{1\prec n}_{k-1\mapsto n}$, $\A^{1\prec\, n-1}_{k\mapsto n-1}$, and $\A^{1\prec\,n-1}_{k-1\mapsto n-1}$. 

First, for $\sigma\in \A^{1\prec n}_{k-1\mapsto n}$, we let $\tau_1 =\phi_1(\sigma)$ be the permutation obtained by swapping the entries at positions $k-1$ and $k$ in $\sigma$. Thus $\tau_1(k-1)=\sigma(k)$ and $\tau_1(k)=\sigma(k-1)=n$. Note that if $k<n$, then $\tau_1(k-1)>\tau_1(k+1)$ (since $\sigma$ avoids $1423$). Clearly, $\tau_1\in\A^{1\prec n}_{k\mapsto n}$.

For $\sigma\in \A^{1\prec n-1}_{k\mapsto n-1}$ and $k<n$, we let $\tau_2 =\phi_2(\sigma)$ be the permutation in $\A^{1\prec n}_{k\mapsto n}$ obtained from $\sigma$ by inserting $n$ at position $k$. That is, $\tau_2(i)=\sigma(i)$ for $i<k$, $\tau_2(k)=n$, and $\tau_2(j)=\sigma(j-1)$ for $j>k$. In particular, $\tau_2(k+1)=n-1$, so $\phi_1(\A^{1\prec n}_{k-1\mapsto n})$ and $\phi_2(\A^{1\prec n-1}_{k\mapsto n-1})$ are disjoint.

Finally, for $\sigma\in \A^{1\prec n-1}_{k-1\mapsto n-1}$ we define the map $\phi_3$ as follows. If $\sigma(n-1)=n-1$, we let $\phi_3(\sigma) = \ins^1_{n-1}(\sigma)$, i.e., the permutation obtained by inserting $1$ at position $n-1$. Observe that this type of permutation is not in the image of $\phi_1$ or $\phi_2$. Moreover, 
\[ \phi_1(\A^{1\prec n}_{n-1\mapsto n}) \cup \phi_3(\A^{1\prec n-1}_{n-1\mapsto n-1})  \subseteq  \A^{1\prec n}_{n\mapsto n}, \]
and every $\tau\in  \A^{1\prec n}_{n\mapsto n}$ is either in the image of $\phi_1$ (if $\tau(n-1)\not=1$) or in the image of $\phi_3$ (if $\tau(n-1)=1$). Thus the above inclusion is an equality, giving the recurrence for $a_{n,n}$.

Suppose now that $k-1<n-1$ and let $m=\sigma(k)$. If there were positions $i_1<i_2<\sigma^{-1}(1)$ such that $\sigma(i_1) < m < \sigma(i_2)$, then $(\sigma(i_1),\sigma(i_2),1,m)$ would form a $2413$ pattern, and if there were positions $\sigma^{-1}(1)<i_3<i_4<k$ such that $\sigma(i_3) > m > \sigma(i_4)$, then $(1,\sigma(i_3),\sigma(i_4),m)$ would form a $1423$ pattern. Therefore, $\sigma$ must be of the form depicted in Figure~\ref{fig:map_phi3}(a), where every box represents a word (possibly empty) avoiding $(1423,2413)$. In fact, $\varrho$ avoids $312$ and the word to the right of $m$ must be decreasing. 

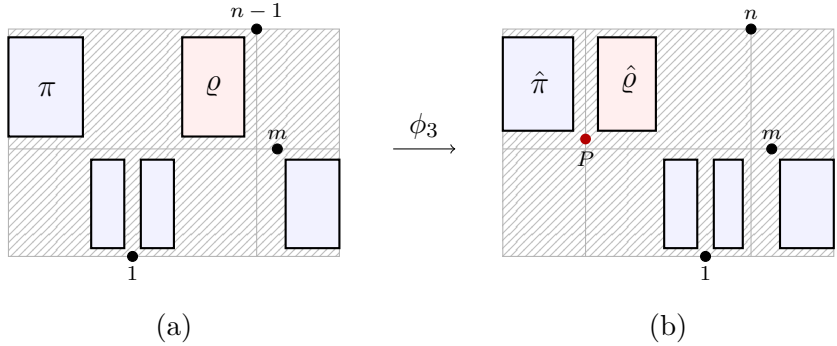
\begin{figure}[ht]
\begin{tikzpicture}[scale=1.1]
\begin{scope}
\draw[mesh] (0,0.25) rectangle (4,3);
\draw[gray!60] (0,1.55)--(4,1.55); 
\draw[gray!60] (3,0.25)--(3,3); 
\draw[thick,fill=blue!5] (0,1.7) rectangle (0.9,2.9);
\draw[thick,fill=blue!5] (1,0.35) rectangle (1.4,1.42);
\draw[thick,fill=blue!5] (1.6,0.35) rectangle (2,1.42);
\draw[thick,fill=pink!25] (2.1,1.7) rectangle (2.85,2.9);
\draw[thick,fill=blue!5] (3.35,0.35) rectangle (4,1.42);
\draw[fill] (1.5,0.25) circle(0.06);
\node[below] at (1.5,0.25) {\scriptsize $1$};
\draw[fill] (3,3) circle(0.06);
\node[above] at (3,3) {\scriptsize $n-1$};
\draw[fill] (3.25,1.55) circle(0.06);
\node[above] at (3.25,1.55) {\scriptsize $m$};
\node[above=2pt] at (0.47,2) {\large $\pi$};
\node[above=2pt] at (2.47,2) {\large $\varrho$};
\node[below=10pt] at (2,0) {(a)};
\end{scope}
\draw[->] (4.65,1.55) -- node[above, midway] {$\phi_3$} (5.4,1.55);
\begin{scope}[xshift=170]
\draw[mesh] (0,0.25) rectangle (4,3);
\draw[gray!60] (0,1.55)--(4,1.55); 
\draw[gray!60] (1,0.25)--(1,3); 
\draw[gray!60] (3,0.25)--(3,3); 
\draw[thick,fill=blue!5] (0,1.77) rectangle (0.85,2.9);
\draw[thick,fill=blue!5] (1.95,0.35) rectangle (2.35,1.42);
\draw[thick,fill=blue!5] (2.55,0.35) rectangle (2.9,1.42);
\draw[thick,fill=pink!25] (1.15,1.77) rectangle (1.85,2.9);
\draw[thick,fill=blue!5] (3.35,0.35) rectangle (4,1.42);
\draw[fill] (2.45,0.25) circle(0.06);
\node[below] at (2.45,0.25) {\scriptsize $1$};
\draw[fill] (3,3) circle(0.06);
\node[above] at (3,3) {\scriptsize $n$};
\draw[fill] (3.25,1.55) circle(0.06);
\node[above] at (3.25,1.55) {\scriptsize $m$};
\fill[red!70!black] (1,1.67) circle(0.065);
\node[below=1pt] at (1,1.67) {\scriptsize $P$};
\node[above] at (0.45,2.1) {\large $\hat\pi$};
\node[above] at (1.52,2.1) {\large $\hat\varrho$};
\node[below=10pt] at (2,0) {(b)};
\end{scope}
\end{tikzpicture}
\caption{Graphical representation of $\phi_3$.}
\label{fig:map_phi3}
\end{figure}

Let $j$ be the position of the rightmost element of $\pi$. We define $\phi_3(\sigma)$ to be the permutation obtained from $\sigma$ by inserting $m+1$ at position $j+1$, and moving $\varrho$ as shown in Figure~\ref{fig:map_phi3}(b), where $P$ is the point with coordinates $(j+1,m+1)$. The words $\hat\pi$ and $\hat\varrho$ are vertical shifts (by one unit) of $\pi$ and $\varrho$. It is easy to verify that this operation does not create any of the forbidden patterns $1423$ or $2413$, so $\phi_3(\sigma)\in \A^{1\prec n}_{k\mapsto n}$.

Since $m$ is at most $n-2$, the permutation $\tau_3=\phi_3(\sigma)$ is not in the image of $\phi_2$. Moreover, by construction, $\tau_3(k-1)<\tau_3(k+1)$, so $\tau_3$ is not in the image of $\phi_1$ either. Thus the images of the injective maps $\phi_1$, $\phi_2$, and $\phi_3$ are disjoint, and we have
\[ \phi_1(\A^{1\prec n}_{k-1\mapsto n})\cup \phi_2(\A^{1\prec n-1}_{k\mapsto n-1}) 
  \cup \phi_3(\A^{1\prec n-1}_{k-1\mapsto n-1})  \subseteq  \A^{1\prec n}_{k\mapsto n}. \]
What about surjectivity? Let $\tau\in\A^{1\prec n}_{k\mapsto n}$. For $3\le k<n$, either $\tau(k-1)>\tau(k+1)$ or $\tau(k-1)<\tau(k+1)$. In the first case, $\tau$ is in the image of the map $\phi_1$. In the latter case, $\tau$ is in the image of $\phi_2$ if $\tau(k+1)=n-1$ or in the image of $\phi_3$ if $\tau(k+1)<n-1$.

In conclusion, the above inclusion is indeed an equality.
\end{proof}

\begin{table}[ht]
\begin{tabular}{c|rrrrrrr||r}
\rule[-5pt]{0pt}{0pt} $n \backslash k$ & 2 & 3 & 4 & 5 & 6 & 7 & 8& $\Sigma$ \\ \hline
2 & 1 & & & & & & & \rule{0pt}{12pt} 1 \\
3 & 1 & 2 & & & & & & 3 \\
4 & 1 & 4 & 6 & & & & & 11 \\
5 & 1 & 6 & 16 & 22 & & & & 45 \\
6 & 1 & 8 & 30 & 68 & 90 & & & 197 \\
7 & 1 & 10 & 48 & 146 & 304 & 394 & & 903 \\
8 & 1 & 12 & 70 & 264 & 714 & 1412 & 1806 & 4279 
\end{tabular}
\bigskip
\caption{Triangle for $\{a_{n,k}\}$ from Proposition~\ref{prop:1423_2413by1nDist}.}
\label{tab:1423_2413by1nDist}
\end{table}

\begin{rem}
The triangular array in Table~\ref{tab:1423_2413by1nDist} is listed in the OEIS \cite[A033877]{Sloane} (with slightly shifted indexing) and has a known generating function; see sections 5-7 of Kreweras' work~\cite{Kreweras73}. 
\end{rem}

\begin{corollary} \label{cor:1423_2413by1nDist}
If $a_{n,k} = \abs{\A^{1\prec n}_{k\mapsto n}}$ and $g(x,t)=\sum\limits_{n=2}^\infty\sum\limits_{k=2}^n\, a_{n,k}\, t^{k} x^n$, then
\[ g(x,t) = \frac{xt^2(S(xt)-1)}{1+t - S(xt)},  \]
where $S(x) = \tfrac12 \big(3 - x - \sqrt{x^2 - 6x + 1}\,\big)$.
\end{corollary}

In particular, $\Av_n^{1\prec n}(1423,2413)$ is counted by the function $g(x,1) = \frac{x(S(x)-1)}{2 - S(x)}$ which gives the little Schr\"oder numbers $1, 3, 11, 45, 197, 903,\dots$. 

\begin{rem}
It is easy to prove that a permutation in $\Av_n(2413)$ is skew-indecomposable if and only if entry $1$ is to the left of entry $n$ (in one-line notation). In addition, as we argued in Corollary~\ref{cor:Schroeder1324_1423}, since the patterns $1423$ and $2413$ are both skew-indecomposable, the above results imply that the class $\Av_n(1423,2413)$ is enumerated by the large Schr\"oder numbers.
\end{rem}

\begin{rem}
Observe that our results reveal the property that $\Av_n(1423,2413)$ has as many elements with the $1$ to the left of $n$ as elements with the $1$ to the right of $n$. The same property holds for the class of separable permutations.
\end{rem}

\section{(1324,2134)-avoiding permutations}
\label{sec:1324_2134}

In this final section, we focus on a different type of positional refinement: we will count the elements of $\Av_{n}(1324,2134)$ by their entry at position $n$. The fact that this class is enumerated by the large Schr\"oder numbers was established by Kremer~\cite{Kremer2000} using generating trees; here we provide an alternative proof.

The decomposition used below shares the spirit of insertion encoding~\cite{ALR05}, but tracks the value of the last entry, a positional refinement that goes beyond total enumeration.

Let 
\[ \A_{n,\ell} = \{\sigma\in\Av_n(1324,2134): \sigma(n)=\ell\}. \]

\begin{proposition}
\label{prop:1324_2134byLast}
Let $1\le \ell\le n$. If $s_{n,\ell} = \abs{\A_{n,\ell}}$, then
\begin{align*}
 s_{1,1} &=1, \;\; s_{2,1}=s_{2,2}=1\\ 
 s_{n,1} &= s_{n,2} = s_{n,3} = \sum_{m=1}^{n-1} s_{n-1,m} \;\text{ for } n\ge 3,\\
 s_{n,\ell} &= 2s_{n-1,\ell-1} + \sum_{m=\ell}^{n-1} s_{n-1,m} \;\text{ for } 4\le \ell\le n.
\end{align*}
\end{proposition}
\begin{proof}
The cases $n=1,2$ are obvious. For $n\ge 3$ and $i\in\{1,2,3\}$, it can be easily seen that every permutation in $\A_{n,i}$ can be uniquely obtained from one in $\Av_{n-1}(1324,2134)$ by inserting $i$ at position $n$. By definition, the set $\Av_{n-1}(1324,2134)$ has a total of $\sum\limits_{m=1}^{n-1} s_{n-1,m}$ elements, hence the claimed formulas for $s_{n,1}$, $s_{n,2}$, and $s_{n,3}$ hold.

For $\ell\ge 4$, the elements of $\A_{n,\ell}$ split naturally into two subsets depending on whether their second to last entry is larger or smaller than $\ell$. On the one hand, every $\sigma\in \A_{n,\ell}$ with $\sigma(n-1)>\ell$ can be uniquely obtained from one in $\A_{n-1,m}$, for some $m\ge \ell$, by inserting $\ell$ at position $n$. This insertion creates none of the forbidden patterns, and the corresponding element of $\A_{n-1,m}$ can be recovered from $\sigma$ just by removing its last entry. In other words, there are $\sum\limits_{m=\ell}^{n-1} s_{n-1,m}$ such permutations.

On the other hand, the set of permutations $\sigma\in \A_{n,\ell}$ with $\sigma(n-1)<\ell$ (i.e.\ ending with an ascent) can in turn be written as the disjoint union of the sets
\begin{align*}
 \U_{n,\ell} &= \{ \sigma\in \A_{n,\ell}: \sigma(n-1) = 1\}, \\
 \V_{n,\ell} &= \{ \sigma\in \A_{n,\ell}: 1<\sigma(n-1) <\ell\}.
\end{align*}
Insertion of $1$ at position $n-1$ gives a clear bijection from $\A_{n-1,\ell-1}$ to $\U_{n,\ell}$. Therefore, we have $\abs{\U_{n,\ell}}=s_{n-1,\ell-1}$. Moreover, the elements of $\V_{n,\ell}$ can be uniquely constructed via the bijective map $\alpha:\A_{n-1,\ell-1}\to \V_{n,\ell}$ defined as follows. 

Let $\tau\in \A_{n-1,\ell-1}$ and let $i$ and $j$ be such that $\tau(i)=1$ and $\tau(j)=2$. If $j<i$, then entry $2$ is to left of $1$, and the entries to the right of $1$ form a decreasing sequence (since $\tau$ avoids the pattern $2134$). In this case, we let $\sigma=\alpha(\tau)$ be the permutation obtained by inserting $2$ into $\tau$ at position $n-1$. By construction, $\sigma(n)=\ell$, and it can be easily verified that this insertion does not create a pattern $1324$ or $2134$.

Now, if $j>i$, then $\tau$ must be of the form $\tau=\perm{\pi,1,\theta,2,\delta,(\ell-1)}$, with possibly empty words $\pi$, $\theta$, and $\delta$. Since $\tau$ avoids $1324$, the entries of $\theta$ (if any) must all be greater than $\ell-1$, and the entries of $\delta$   that are less than $\ell-1$ (if any) must form an increasing sequence. If $\pi$ is empty, or if its entries are all larger than $\ell-1$, we define $\alpha(\tau)$ as the permutation obtained from $\tau$ by inserting $\ell-1$ at position $n-1$. The permutation $\alpha(\tau)$ ends with the ascent $(\ell-1)\,\ell$ and belongs to $\V_{n,\ell}$.

Finally, if $\tau=\perm{\pi,1,\theta,2,\delta,(\ell-1)}$ and $\pi$ has an entry less than $\ell-1$, we let $m=\min(\pi)$. Note that if $m\le c<\ell-1$, then $c$ must be contained in $\pi$. We define $\alpha(\tau)$ as the permutation obtained from $\tau$ by inserting $m$ at position $n-1$. In other words, $\alpha(\tau) = \hat\pi\,1\,\hat\theta\,2\,\tilde\delta\,m\,\ell$, where $m<\ell-1$, $\hat\pi$ and $\hat\theta$ are vertical shifts (by one) of $\pi$ and $\theta$, and $\tilde\delta$ is obtained from $\delta$ by increasing any entry greater than $m$ by one. Since the entries of $\hat\pi$ and $\hat\theta$ are larger than $m$, and since the entries of $\tilde\delta$ less than $\ell$ (if any) form an increasing sequence, $m$ is not part of a $2134$ pattern. Moreover, since every $d$ with $m<d<\ell$ is contained in $\hat\pi$, entry $m$ cannot be part of a $1324$ pattern either. Therefore, $\alpha(\tau)\in\V_{n,\ell}$.

Observe that the permutations obtained by the above process (when $j>i$) don't have the $2$ in the second to last position, so there is no overlap with the process when $j<i$. 

The map $\alpha$ is reversible, hence $\abs{\V_{n,\ell}}=s_{n-1,\ell-1}$. In conclusion, there are $2 s_{n-1,\ell-1}$ permutations in $\A_{n,\ell}$ ending with an ascent, and the recurrence for $s_{n,\ell}$ holds.
\end{proof}

\begin{table}[t]
\begin{tabular}{c|rrrrrrr||r}
\rule[-5pt]{0pt}{0pt} $n \backslash \ell$ & 1& 2 & 3 & 4 & 5 & 6 & 7 & $\Sigma$ \\ \hline
1 & 1 & & & & & & &\rule{0pt}{12pt} 1 \\
2 & 1 & 1 & & & & &  & 2\\
3 & 2 & 2 & 2 & & & & & 6 \\
4 & 6 & 6 & 6 & 4 & & & & 22 \\
5 & 22 & 22 & 22 & 16 & 8 & & & 90 \\
6 & 90 & 90 & 90 & 68 & 40 & 16 & & 394 \\
7 & 394 & 394 & 394 & 304 & 192 & 96 & 32 & 1806
\end{tabular}
\bigskip
\caption{Triangle for $\{s_{n,\ell}\}$ from Proposition~\ref{prop:1324_2134byLast}.}
\label{tab:1324_2134byLast}
\end{table}

\begin{rem}
The triangular array in Table~\ref{tab:1324_2134byLast} is the reverse of \oeis{A341695}, which can be found in work by Lin and Kim~\cite[Section~3]{LinKim21} in the context of inversion sequences, and in work by Mansour and Shattuck~\cite{MansourShattuck23} as the distribution of the first letter statistic on the class of $(1243,1324)$-avoiding permutations.
\end{rem}

\begin{corollary}
The generating function $h(x,u) = \sum\limits_{n=1}^{\infty} \sum\limits_{\ell=1}^{n} s_{n,\ell}\, u^\ell\, x^n$ is given by
\[ h(x,u) = \frac{2ux(1-u)(1-ux) + ux\big(1 - u(1-u)x\big)\big(1-x-\sqrt{1-6x+x^2}\big)}{2\big(1-u(1+x)+2u^2 x\big)}. \]
In particular,
\[ h(x,1) = \frac{1-x-\sqrt{1-6x+x^2}}{2},\]
hence $|\Av_n(1324,2134)|$ is counted by the large Schr\"oder numbers.
\end{corollary}
\begin{proof}
Let $H_n(u) = \sum\limits_{\ell=1}^{n} s_{n,\ell}\, u^\ell$ and $r_n=H_n(1)$. Clearly, $H_1(u) = u$ and $H_2(u) = u+u^2$. Moreover, using the recurrence relation for $s_{n,\ell}$, one derives the functional equation
\[ (1-u)\, H_n(u) = u(1-2u)\, H_{n-1}(u) + r_{n-1}\, u - r_{n-2}\, u^2(1-u) \;\text{ for } n\ge 3. \]
Since $h(x,u) = \sum\limits_{n=1}^\infty H_n(u)\, x^n$ and $h(x,1) = \sum\limits_{n=1}^\infty r_n\, x^n$, the above functional equation and routine algebraic manipulations give
\[ \bigl[1 - u(1+x) + 2u^2 x\bigr]\, h(x,u) = ux(1-u)(1-ux) + ux\bigl(1 - u(1-u)x\bigr)\, h(x,1). \]
Letting $1 - u(1+x) + 2u^2 x = 0$, we get $u = \frac{1+x-\sqrt{1-6x+x^2}}{4x}$, and so
\[ h(x,1) = \frac{1-x-\sqrt{1-6x+x^2}}{2}. \]
This leads to the claimed formula for $h(x,u)$.
\end{proof}

\subsection{Positional statistics for $1\prec n$}
We finish the section with a related result:
\[ \abs{\Av_{n}^{1\prec n}(1324,2134)} = \binom{2n-3}{n-1} \;\text{ for } n\ge 2. \]
Observe that $\sigma\in \Av_{n}^{1\prec n}(1324,2134)$ if and only if $\sigma^{rc} \in \Av_{n}^{1\prec n}(1243,1324)$. Since the structure of $(1243,1324)$-avoiding permutations with $1\prec n$ is more amenable to decomposition, we will prove the above formula for $\Av_{n}^{1\prec n}(1243,1324)$ instead.

Recall that $\Av_{n,k}^{1\prec n}(1243,1324)$ denotes the set of permutations in $\Av_n(1243,1324)$ having the $1$ to the left of $n$ at distance $k$. We will focus on these sets for $1\le k\le n-1$, starting with the enumeration of $\Av_{n,1}^{1\prec n}(1243,1324)$.

\medskip
For $n\ge 2$ and $\ell\in\{1,\dots,n-1\}$, let
\[  \T_{n,\ell} = \{\sigma\in\Av_{n,1}^{1\prec n}(1243,1324): \sigma(1)=\ell\}, \]
and let $t_{n,\ell}=\abs{\T_{n,\ell}}$. Clearly, $\Av_{n,1}^{1\prec n}(1243,1324) = \bigcup\limits_{\ell=1}^{n-1} \T_{n,\ell}$. Let $s_{n} = \sum\limits_{\ell=1}^{n-1} t_{n,\ell}$.

\begin{lemma} 
For $n\ge 3$, we have
\[ t_{n,1} = t_{n,2} = 2^{n-3} \;\text{ and }\; t_{n,n-2} = t_{n,n-1} = \sum_{\ell=1}^{n-2} t_{n-1,\ell}. \]
\end{lemma}
\begin{proof}
Every element of $\T_{n,1}$ is of the form $\sigma = 1\,n\,\pi$, where $\red(\pi)\in \Av_{n-2}(132,213)$. Hence $t_{n,1} = 2^{n-3}$. Moreover, every permutation in $\T_{n,2}$ is of the form $21n\,\pi$ or $2\,\tau1n\,\pi$ with a nonempty word $\tau$. Every permutation of the form $21n\,\pi$ can be obtained from an element of $\T_{n-1,1}$ by inserting 2 at position 1. Now, if $\sigma = 2\,\tau1n\,\pi$ and $n\ge 4$, entry $n-1$ in $\sigma$ must be adjacent to the left of 1. Otherwise, it would create a $1324$ pattern (if it is in $\tau$ but not adjacent to 1) or a $1243$ pattern (if it is in $\pi$). Every $\sigma$ of this type can be obtained from an element of $\T_{n-1,2}$ by inserting $n-1$ to the immediate left of 1. In conclusion, we have 
\[ t_{3,2} = 1 \;\text{ and }\; t_{n,2} = t_{n-1,1} + t_{n-1,2} \;\text{ for } n\ge 4. \]
This implies $t_{n,2} - t_{n-1,2} = 2^{n-4}$, which leads to $t_{n,2} = 2^{n-3}$.

On the other hand, it can be easily verified that every permutation in $\T_{n,n-i}$ for $i=1$ or $i=2$ can be obtained from a unique permutation in $\T_{n-1,\ell}$, $1\le\ell\le n-2$, by inserting $n-i$ at position 1. Thus, for $i\in\{1,2\}$, we have $t_{n,n-i} = \sum\limits_{\ell=1}^{n-2} t_{n-1,\ell}$.
\end{proof}

\begin{table}[ht]
\begin{tabular}{c|rrrrrrr||r}
\rule[-5pt]{0pt}{0pt} $n \backslash \ell$ & 1 & 2 & 3 & 4 & 5 & 6 & 7 & $\Sigma$ \\ \hline
2 & 1 & & & & & & & \rule{0pt}{12pt} 1 \\
3 & 1 & 1 & & & & & & 2 \\
4 & 2 & 2 & 2 & & & & & 6 \\
5 & 4 & 4 & 6 & 6 & & & & 20 \\
6 & 8 & 8 & 14 & 20 & 20 & & & 70 \\
7 & 16 & 16 & 30 & 50 & 70 & 70 & & 252 \\
8 & 32 & 32 & 62 & 112 & 182 & 252 & 252 & 924 
\end{tabular}
\bigskip
\caption{Triangle for $\{t_{n,\ell}\}$ (reverse of \cite[A171698]{Sloane}).}
\end{table}

\begin{lemma} 
For $3\le \ell\le n-2$, we have
\[ t_{n,\ell} = t_{n,\ell-1} + t_{n-1,\ell}. \]
\end{lemma}
\begin{proof}
Recall that the elements of $\T_{n,\ell}$ start with $\ell$, have the $1$ adjacent to the left of $n$, and avoid the patterns $1243$ and $1324$. We start by splitting $\T_{n,\ell}$ into three disjoint subsets:
\begin{align*}
 \T_1 &= \{\sigma\in \T_{n,\ell}: \sigma^{-1}(\ell-1) > \sigma^{-1}(n)\}, \\
 \T_2 &= \{\sigma\in \T_{n,\ell}: \sigma^{-1}(\ell-1)<\sigma^{-1}(m)\le\sigma^{-1}(n)\text{ for every } m>\ell\}, \\
 \T_3 &= \T_{n,\ell}\setminus(\T_1\cup \T_2).
\end{align*}
In other words, the elements of $\T_1$ have entry $\ell-1$ to the right of $n$, while the elements of $\T_2$ must be of the form $\perm{\ell,\tau,(\ell-1),\theta,1,n,\pi}$, where $\tau$ and $\pi$ are either empty or have entries smaller than $\ell$. In this case, each entry $m$ with $\ell<m<n$ is contained in $\theta$.

Every element of $\T_1\cup \T_2$ can be uniquely obtained from one of $\T_{n,\ell-1}$ by swapping the entries $\ell-1$ and $\ell$. Thus $\abs{\T_1\cup \T_2} = t_{n,\ell-1}$.

On the other hand, every $\sigma\in \T_3$ must be of the form $\sigma=\perm{\ell,\tau,(\ell-1),\theta,1,n,\pi}$, where $\tau$ or $\pi$ have at least one entry greater than $\ell$. If $\tau$ is empty, then $\pi$ must have an element $k>\ell$, and $\ell+1$ must be contained in $\pi$ (otherwise, it would be in $\theta$, and $(\ell-1,\ell+1,n,k)$ would form a $1243$ pattern). We let $\sigma'\in\T_{n-1,\ell}$ be the permutation obtained from $\sigma$ by removing entries $\ell$ and $\ell-1$ (thus each entry greater than $\ell$ goes down by 2), and inserting $\ell$ back into position 1. Note that $\sigma'$ now has entry $\ell-1$ to the right of $n-1$.

If $\tau$ is nonempty, then it must contain an entry larger than $\ell$, say $m$. If $\sigma(2)<\ell$, then $(\sigma(2),m,\ell-1,n)$ would form a forbidden $1324$ pattern. Thus $\sigma(2)$ must be greater than $\ell$ and greater than all entries in $\pi$ (to avoid a $1243$ pattern). We let $\sigma'\in\T_{n-1,\ell}$ be the permutation obtained from $\sigma$ by removing $\sigma(2)$. In this case, the resulting permutation $\sigma'$ has entry $\ell-1$ to the left of $n-1$, so there are no duplicates with the case when $\tau$ is empty. 

The inverse map from $\T_{n-1,\ell}$ to $\T_3$ is straightforward. If $\sigma'\in\T_{n-1,\ell}$ has entry $\ell-1$ to the right of $n-1$, we build $\sigma$ by inserting entry $\ell-1$ into $\sigma'$ at position 2 while keeping $\ell$ in position 1. Now, if $\sigma'\in\T_{n-1,\ell}$ has entry $\ell-1$ to the left of $n-1$, we insert $\max(\ell,\max(\pi'))+1$ into $\sigma'$ at position 2, where $\pi'$ is the word (possibly empty) to the right of $n-1$ in $\sigma'$.

In conclusion, $\abs{\T_3} = t_{n-1,\ell}$, and we arrive at $t_{n,\ell} = t_{n,\ell-1} + t_{n-1,\ell}$. 
\end{proof}

\begin{proposition} \label{prop:centralBinom}
The generating function $g(x,u) = \sum\limits_{n=2}^{\infty} \sum\limits_{\ell=1}^{n-1} t_{n,\ell}\, u^\ell\, x^n$ is given by
\[ g(x,u) = \frac{1}{1-u-x}\left[\frac{u(1-u)\,x^2(1-x)^2}{1-2x} - \frac{u^3 x^3}{\sqrt{1-4ux}}\right]. \]
In particular, 
\[ g(x,1) = \frac{x^2}{\sqrt{1-4x}} = \sum_{n=2}^\infty \binom{2(n-2)}{n-2} x^n, \]
and therefore $s_{n} = \abs{\Av_{n,1}^{1\prec n}(1243,1324)} = \binom{2(n-2)}{n-2}$.
\end{proposition}
\begin{proof}
Let $G_n(u) = \sum\limits_{\ell=1}^{n-1} t_{n,\ell}\, u^\ell$. Clearly, $G_2(u) = u$ and $G_3(u) = u+u^2$. Moreover, using the above two lemmas and the notation $s_{n} = \sum\limits_{\ell=1}^{n-1} t_{n,\ell}$, one derives the functional equation
\[ (1-u) G_n(u) = G_{n-1}(u) + u(1-u) 2^{n-4} - s_{n-1} u^n \;\text{ for } n\ge 4. \]
Since $g(x,u) = \sum\limits_{n=2}^\infty G_n(u) x^n$ and $g(x,1) = \sum\limits_{n=2}^\infty s_{n} x^n$, we get
\[ (1-u-x) g(x,u) =  \frac{u(1-u)x^2(1-x)^2}{1-2x} - ux g(ux,1). \]
The kernel method (letting $u = 1-x$ and $z = (1-x)x$) then provides
\[ g(z,1) = \frac{z^2}{\sqrt{1-4z}}. \]
This leads to the claimed formula for $g(x,u)$. 
\end{proof}

\begin{proposition}
If $a_{n,k} = \abs{\Av_{n,k}^{1\prec n}(1243,1324)}$, then for $n\ge 2$,
\begin{gather*}
 a_{n,1} = s_n = \tbinom{2n-4}{n-2}, \;\; a_{n,n-1} = 1, \text{ and} \\[5pt]
 a_{n,k} = a_{n,k+1} + a_{n-1,k-1} \;\text{ for }\; 2\le k\le n-2.
\end{gather*}
Therefore, $a_{n,k} = \binom{2n-k-3}{n-2}$ for $1\le k\le n-1$. This gives the triangle \cite[A092392]{Sloane}. 
\end{proposition}
\begin{proof}
The statement for $a_{n,1}$ was proved in Proposition~\ref{prop:centralBinom}. Now, if a permutation starts with 1, ends with $n$, and avoids $1243$ and $1324$, it must be the identity. So, $a_{n,n-1}=1$.

For $2\le k\le n-2$, every element of $\Av_{n,k}^{1\prec n}(1243,1324)$ must be of the form $\sigma = \perm{\tau,1,\theta,n,\pi}$, where $\theta$ is nonempty and increasing. Since $\sigma$ avoids $1243$, entry $n-1$ must be either in $\tau$ or in $\theta$. 
Let $\A_\tau$ and $\A_\theta$ be the disjoint subsets of $\Av_{n,k}^{1\prec n}(1243,1324)$ according to the corresponding position of $n-1$. Clearly, $a_{n,k} = \abs{\A_\tau} + \abs{\A_\theta}$.

Every $\sigma\in\A_\theta$ can be uniquely obtained from a permutation $\sigma'\in\Av_{n-1,k-1}^{1\prec n-1}(1243,1324)$ by inserting $n$ into $\sigma'$ to the immediate right of $n-1$. Therefore, $\abs{\A_\theta} = a_{n-1,k-1}$.

On the other hand, every $\sigma\in \A_\tau$ must be of the form $\sigma = \perm{\tau_1,(n-1),\tau_2,1,\theta,n,\pi}$ with possibly empty words $\tau_1$ and $\tau_2$. Let $m = n-1-\abs{\tau_1}$ and let $\hat\tau_1$ be the word obtained from $\tau_1$ by increasing its entries by one. The map
\[ \sigma = \perm{\tau_1,(n-1),\tau_2,1,\theta,n,\pi} \mapsto \hat\tau_1\,\tau_2\,1\,\theta\, m\,n\,\pi \]
gives a bijection between $\A_\tau$ and $\Av_{n,k+1}^{1\prec n}(1243,1324)$, hence $\abs{\A_\tau} = a_{n,k+1}$.
\end{proof}

\begin{corollary}
For $n\ge 2$, we have
\[ \abs{\Av_{n}^{1\prec n}(1243,1324)} = \sum_{k=1}^{n-1} a_{n,k} = \binom{2n-3}{n-1}. \]
\end{corollary}

\bigskip

\subsection*{Statements \& Declarations}
The authors declare that no funds, grants, or other financial support were received during the preparation of this manuscript. The authors have no competing interests to disclose.


\end{document}